\documentclass{article}

\usepackage{amsmath,amsfonts,amssymb,latexsym,graphics,epsfig,url}
\usepackage{amsthm}

\newtheorem{Theorem}{Theorem}[section]
\newtheorem{Lemma}[Theorem]{Lemma}

\numberwithin{equation}{section}

\def\Par{\pi}

\def\j{\mbox{\boldmath $j$}}

\def\s{\mbox{\boldmath $s$}}

\def\vec0{\mbox{\boldmath $0$}}

\def\A{\mbox{\boldmath $A$}}
\def\B{\mbox{\boldmath $B$}}
\def\C{\mbox{\boldmath $C$}}

\def\I{\mbox{\boldmath $I$}}

\def\S{\mbox{\boldmath $S$}}

\begin{document}

\title{A note on the order of iterated line digraphs}

\author{C. Dalf\'o$^a$, M.A. Fiol$^b$\\
$^{a}${\small Departament de Matem\`atiques} \\
{\small Universitat Polit\`ecnica de Catalunya} \\
{\small Barcelona, Catalonia} \\
{\small {\tt{cristina.dalfo@upc.edu}}} \\
$^{b}${\small Departament de Matem\`atiques} \\
{\small Universitat Polit\`ecnica de Catalunya} \\
{\small Barcelona Graduate School of Mathematics} \\
{\small Barcelona, Catalonia} \\
{\small{\tt{fiol@ma4.upc.edu}}}}

\date{}

\maketitle

\begin{abstract}
Given a digraph $G$, we propose a new method to find the recurrence equation for
the number of vertices $n_k$ of the $k$-iterated line digraph
$L^k(G)$, for $k\geq0$, where $L^0(G)=G$. We obtain this result by using the
minimal polynomial of a quotient digraph $\pi(G)$ of $G$. We show some examples of this method applied to the so-called cyclic Kautz, the unicyclic, and the acyclic digraphs. In the first case, our method gives the enumeration of the ternary length-2 squarefree words of any length.
\end{abstract}



%%%%%%%%%%%%%%%%%%%%%%%%%%%%%%%%%%%%%%%%%%%%%%%%%%%%%%%%%%%%%%

\section{Preliminaries}

In this section we recall some basic notation and results
concerning digraphs and their spectra.
A digraph $G=(V,E)$ consists of a (finite) set
$V=V(G)$ of vertices and a set $E=E(G)$ of arcs (directed edges) between vertices
of $G$. As the initial and final vertices of an arc are not necessarily different, the
digraphs may have \emph{loops} (arcs from a vertex to itself),
and \emph{multiple arcs}, that is, there can be more than one arc from each vertex
to any other. If $a=(u,v)$ is an arc from $u$ to  $v$, then vertex $u$ (and arc $a$)
is {\em adjacent to} vertex $v$, and vertex $v$ (and arc $a$) is {\em adjacent from}
$u$. Let $G^+(v)$ and $G^-(v)$ denote the set of arcs adjacent from and to vertex
$v$, respectively.  A digraph $G$ is $d$\emph{-regular} if $|G^+(v)|=|G^-(v)|=d$ for all $v\in V$.

In the line digraph $L(G)$ of a digraph $G$, each vertex of $L(G)$ represents an
arc of $G$, that is, $V(L(G))=\{uv|(u,v)\in E(G)\}$; and vertices $uv$ and $wz$ of $L(G)$ are adjacent if and only if $v=w$, namely, when arc $(u,v)$ is adjacent to arc $(w,z)$ in $G$. For $k\geq0$, we consider the sequence of line digraph iterations
$L^0(G)=G,L(G),L^2(G),\ldots,L^k(G)=L(L^{k-1}(G)),\ldots$
It can be easily seen that every vertex of $L^k(G)$ corresponds to a walk
$v_0,v_1,\ldots,v_k$ of length $k$ in $G$, where $(v_{i-1,},v_{i})\in E$ for $i=1,\ldots,k$.
Then, if there is one arc between pairs of vertices and $\A$ is the adjacency matrix of $G$, the $uv$-entry of the power $\A^k$, denoted by $a_{uv}^{(k)}$, is the number of $k$-walks from vertex $u$ to vertex $v$, and the order $n_k$ of $L^k(G)$ turns out to be
\begin{equation}\label{orderL^kG}
n_k=\j\A^k\j^{\top},
\end{equation}
where $\j$ stands for the all-$1$ vector. If there are multiple arcs between pairs of vertices, then the corresponding entry in the matrix is not 1, but the number of these arcs.
If $G$ is a $d$-regular digraph with $n$ vertices then its line
digraph $L^k(G)$ is
$d$-regular with $n_k=d^kn$ vertices.

Recall also that a digraph $G$ is \emph{strongly connected} if there is a
(directed) walk between every
pair of its vertices. %If the underlying graph $UG$ is connected, then the digraph
%is called \emph{weakly
%connected}. Moreover, it is known that $G$ is strongly connected if and only if its
%line digraph $L(G)$ is strongly connected.
If $G$ is strongly connected, different from a directed cycle, and it has
diameter $D$, then its line digraph $L^k(G)$ has diameter $D+k$. See Fiol, Yebra, and Alegre~\cite{fya84} for more details.
The interest of the line digraph technique is that it allows us to obtain digraphs with small diameter and large connectivity.
For a comparison between the line digraph technique and other techniques to obtain digraphs with minimum diameter see Miller, Slamin, Ryan and Baskoro~\cite{MiSlRyBa13}. Since these techniques are related to the degree/diameter problem, we refer also to the comprehensive survey on this problem by Miller and \v{S}ir\'{a}\v{n}~\cite{ms}.

For the concepts and/or results not presented here, we refer the reader to some of
the basic textbooks and papers on the subject; about digraphs see, for instance,
Chartrand and Lesniak~\cite{cl96} or Diestel~\cite{d10}, and Godsil~\cite{g93} about the quotient graphs.

%The reason for obtaining `small' diameters is that, given a degree and a number of %vertices, if
%$D^*=D^*(G)=D^*(d,N)$ is the minimum diameter that $G$ can reach according to the
%Moore bound for digraphs (see the comprehensive survey by Miller and \v{S}ir\'a\v{n}~\cite{ms}), then the minimum
%diameter of the $k$-iterated line digraph $L^k(G)$ is
%$D^*(L^k(G))=D^*(d,d^kN)=D^*+t$. So, if
%$D(G)=D^*+\epsilon$ for a non-negative integer $\epsilon$ (``excess''), then
%$L^k(G)$ has diameter $D^*+t+\epsilon=D^*(L^k(G))+\epsilon$, with
%$$
%\lim_{t\rightarrow \infty}\frac{D^*(L^k(G))}{D(L^k(G))}=\lim_{t\rightarrow
%\infty}\frac{D^*+t}{D^*+t+\epsilon}=1.
%$$

This paper is organized as follows. In Section~\ref{sec:reg-part}, we recall the
definition of regular partitions and we give some lemmas about them.
In Section~\ref{sec:main-result} we prove our main result.
In Section~\ref{sec:examples}, we give examples in which the sequence on the number
of vertices of iterated line digraphs is increasing, tending to a positive
constant, or tending to zero.

\section{Regular partitions}
\label{sec:reg-part}

Let $G$ be a digraph with adjacency matrix $\A$. A partition $\Par=(V_1,\ldots,
V_m)$ of its vertex set $V$ is called {\em regular} (or {\em equitable})
whenever, for any $i,j=1,\ldots,m$, the {\em intersection numbers}
$b_{ij}(u)=|G^+(u)\cap V_j|$, where $u\in V_i$, do not depend on the vertex $u$ but only on the subsets (usually called {\em classes} or {\em
cells}) $V_i$ and $V_j$. In this case, such numbers are simply written as $b_{ij}$,
and the $m\times m$ matrix $\B=(b_{ij})$ is referred to as the {\em quotient matrix} of $\A$ with respect to $\Par$. This is also represented by the {\em quotient (weighted) digraph} $\pi(G)$ (associated to the partition $\pi$),
with vertices representing the cells, and an arc with weight $b_{ij}$ from  vertex
$V_i$ to vertex $V_j$ if and only if $b_{ij}\neq 0$. Of course, if $b_{ii}>0$ for some $i=1,\ldots,m$, the quotient digraph $\pi(G)$ has loops.

The {\em characteristic matrix} of (any) partition $\Par$ is the $n\times m$ matrix
$\S=(s_{ui})$ whose $i$-th column is the characteristic vector of $V_i$, that is, $s_{ui}=1$ if $u\in V_i$, and $s_{ui}=0$ otherwise. In terms of such a matrix,
we have the following characterization of regular partitions.

\begin{Lemma}
Let $G=(V,E)$ be a digraph with adjacency matrix $\A$, and vertex partition $\Par$
with characteristic matrix $\S$. Then $\Par$ is regular if and only if there
exists an $m\times m$ matrix $\C$ such that $\S\C=\A\S$. Moreover, $\C=\B$, the
quotient matrix of $\A$ with respect to $\Par$.
\end{Lemma}

\begin{proof}
Let $\C=(c_{ij})$ be an $m\times m$ matrix. For any fixed  $u\in V_i$ and
$j=1,\ldots,m$, we have
$$
(\S\C)_{uj} = \sum_{k=1}^m s_{uk}c_{kj}=c_{ij}, \quad
(\A\S)_{uj} =  \sum_{v\in V} a_{uv} s_{vj}=|G^+(u)\cap V_j|= b_{ij}(u),
$$
and the result follows.
\end{proof}

%Given a regular partition $\pi=(V_1,\ldots,V_k)$ of a digraph $G$, the
%\emph{quotient digraph} $\pi(G)$ is a (weighted) digraph with the cells of $\pi$
%as its vertices, and with an arc from $V_i$ to $V_j$ having weight $b_{ij}>0$, if any of the vertices in $u\in V_i$ is adjacent to $b_{ij}$ vertices in $V_j$. Of course, if $b_{ii}>0$ for some $i=1,\ldots,m$, the quotient digraph $\pi(G)$ has loops.

Most of the results about regular partitions in graphs can be generalized for
regular partitions in digraphs. For instance, using the above lemma it can be proved that all the eigenvalues of the quotient matrix $\B$ are also eigenvalues of $\A$. Moreover, we have the following result.

\begin{Lemma}
Let $G$ be a digraph with adjacency matrix $\A$. Let $\pi=(V_1,\ldots,$ $V_m)$ be a
regular partition of $G$, with quotient matrix $\B$. Then, the number of $k$-walks from each vertex $u\in V_i$ to all vertices of $V_j$ is the $ij$-entry of $\B^{k}$.
\end{Lemma}

\begin{proof}
We use induction. The result is clearly true for $k=0$, since $\B^0=\I$, and for
$k=1$ because of the definition of $\B$. Suppose that the result holds for some $k>1$. Then the set of walks of length $k+1$ from $u\in V_i$ to the vertices of $V_j$ is in bijective correspondence with the set of $k$-walks  from $u$ to vertices $v\in V_h$ adjacent to some vertex of $V_j$. Then, the number of such
walks is $\sum_{h=1}^m(\B^{k})_{ih}b_{hj}=(\B^{k+1})_{ij}$, as claimed.
\end{proof}

As a consequence of this lemma, the number of vertices of $L^{k}(G)$ is
\begin{equation}
\label{n_l}
n_{k}=\sum_{i=1}^m |V_i|\sum_{j=1}^m (\B^{k})_{ij}=\s\B^{k}\j^{\top},
\end{equation}
where $\s=(|V_1|,\ldots,|V_m|)$  and $\j=(1,\ldots,1)$.

\section{Main result}
\label{sec:main-result}

In the following result, we obtain a recurrence equation on the number of
vertices $n_k$ of the
$k$-iterated line digraph of a digraph $G$.
%, having a regular partition, from the minimal polynomial of its quotient matrix $\B$.

\begin{Theorem}
\label{maintheo}
Let $G=(V,E)$ be a digraph on $n$ vertices, and consider a regular partition
$\pi=(V_1,\ldots,V_m)$ with quotient matrix $\B$. Let $m(x)=x^r-\alpha_{r-1} x^{r-1}-\cdots-\alpha_0$ be the minimal polynomial of $\B$. Then, the number of vertices $n_k$ of the $k$-iterated line digraph $L^k(G)$ satisfies the
recurrence
\begin{equation}
\label{recur}
n_k= \alpha_{r-1} n_{k-1}+\cdots+\alpha_{0} n_{k-r},\qquad k=r,r+1,\ldots
\end{equation}
initialized with the values $n_{k}$, for $k=0,1,\ldots,r-1$, given by \eqref{n_l}.
\end{Theorem}

\begin{proof}
Since the  polynomial $x^{k-r}m(x)$ annihilates $\B$ for any $k\ge 0$,
we have
$$
\B^k= \alpha_{r-1} \B^{k-1}+\cdots+\alpha_{0} \B^{k-r}.
$$
Then, by \eqref{n_l}, we get the recurrence
%$$
%n_k=\s\B^k\j^{\top} = \alpha_{r-1} \s\B^{k-1}\j^{\top}+\cdots+\alpha_{0}
%\s\B^{k-r}\j^{\top}
% =\alpha_{r-1}n_{k-1}+\cdots+\alpha_{0}n_{k-r},
%$$
\begin{align*}
n_k=\s\B^k\j^{\top} &= \alpha_{r-1} \s\B^{k-1}\j^{\top}+\cdots+\alpha_{0}
\s\B^{k-r}\j^{\top}\\
 &=\alpha_{r-1} n_{k-1}+\cdots+\alpha_{0} n_{k-r},
\end{align*}
with the first values $n_{k}$, for $k=0,\ldots, r-1$, given as claimed.
\end{proof}

\section{Examples}
\label{sec:examples}

In what follows, we give examples of the three possible behaviours of the sequence
$n_0, n_1, n_2,\ldots$
Namely, when it is increasing, tending to a positive constant, or tending to zero.

\subsection{Cyclic Kautz digraphs}
\label{sec:ex-CK}

\begin{figure}[t]
    \vskip-.5cm
    \begin{center}
        \includegraphics[width=12cm]{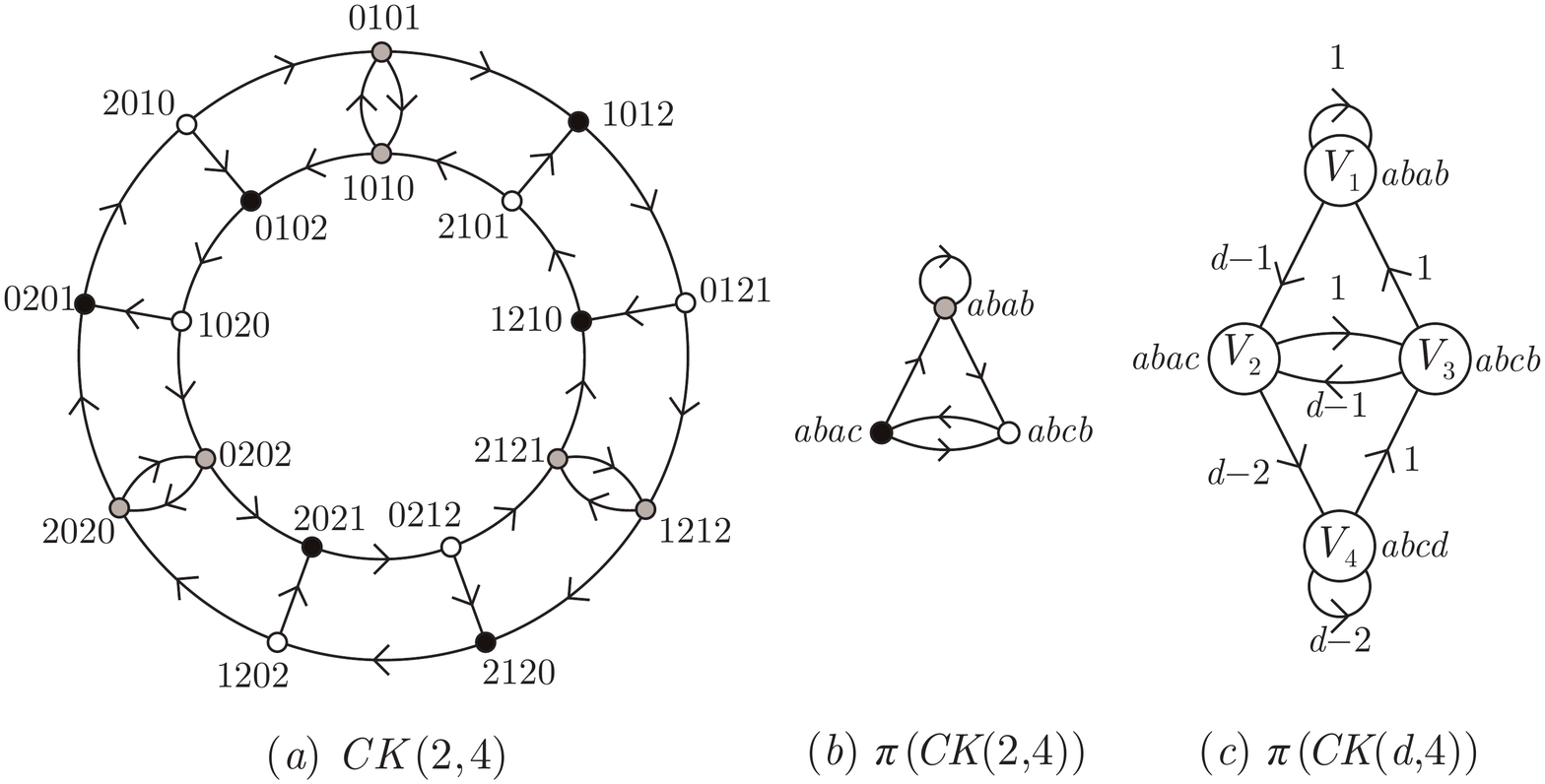}
    \end{center}
    \vskip-2.75cm
	\caption{The cyclic Kautz digraph $CK(2,4)$, its quotient $\pi(CK(2,4))$, and
the quotient digraph of $CK(d,4)$.}
	\label{fig:quocient-donut-color}
\end{figure}

%The Kautz digraph $K(d,\ell)$ has vertices labeled by all possible sequences
%$a_1\ldots a_\ell$ of length
%$\ell$ with different consecutive symbols, $a_i\neq a_{i+1}$ for
%$i=1,\ldots,\ell-1$, from an alphabet
%$\Sigma$ of $d+1$ distinct symbols. Thus, $K(d,\ell)$ is $d$-regular and has
%$(d+1)d^{\ell-1}$ vertices.

The {\em cyclic Kautz digraph} $CK(d,\ell)$, introduced by B\"{o}hmov\'{a},
Dalf\'{o}, and Huemer in~\cite{BoDaHu14}, has vertices labeled by all possible sequences $a_1\ldots a_\ell$  with $a_i\in\{0,1,\ldots,d\}$, $a_i\neq a_{i+1}$ for $i=1,\ldots,\ell-1$, and $a_1\neq a_\ell$. Moreover,
there is an arc from vertex $a_1 a_2\ldots a_\ell$ to vertex $a_2 \ldots a_\ell
a_{\ell+1}$, whenever $a_1\neq a_\ell$ and $a_2\neq a_{\ell+1}$. By this definition, we observe that the cyclic Kautz digraph $CK(d,\ell)$ is a subdigraph of the well-known Kautz digraph $K(d,\ell)$, defined in the same way, but without the requirement $a_1\neq a_\ell$.

For example, Figure~\ref{fig:quocient-donut-color}$(a)$ shows the cyclic Kautz
digraph $CK(2,4)$. Notice that, in general, such digraphs are not $d$-regular and, hence, the number of vertices of their iterated line digraphs are not obtained by repeatedly multiplying by $d$. Instead, we can apply our method, as shown next with $CK(2,4)$. This digraph has a regular partition $\pi$ of its vertex set into three classes (each one with 6 vertices): $abcb$ (the second and the last digits are equal), $abab$ (the first and the third digits are equal, and also the second and the last), and $abac$ (the first and the third digits are equal). Then, the quotient matrix of $\pi$ (which in this case coincides with the adjacency matrix of $\pi(CK(2,4))$) is
$$
\B=\left(
  \begin{array}{ccc}
    0 & 1 & 1 \\
    0 & 1 & 1 \\
    1 & 0 & 0 \\
  \end{array}
\right),
$$
and it has minimal polynomial $m(x)=x^3-x^2-x$. Consequently, by
Theorem~\ref{maintheo}, the number of vertices of $L^k(CK(2,4))$ satisfies the recurrence $n_k=n_{k-1}+n_{k-2}$ for $k\ge 3$. In fact, in this
case, $\s(\B^2-\B-\I)\j^{\top}=0$, and the above recurrence applies from $k=2$.
This, together with the initial values $n_0=18$ and $n_1=\s\B\j^{\top}=30$, yields the Fibonacci sequence, $n_2=48, n_3=78, n_4=126\ldots$,
%with the closed formula
%$$
%n_k=\left(9+\frac{21}{\sqrt{5}}\right)\left(\frac{1+\sqrt{5}}{2}\right)^k
%+\left(9-\frac{21}{\sqrt{5}}\right)
%\left(\frac{1-\sqrt{5}}{2}\right)^k,\qquad k=0,1,\ldots
%$$
as B\"{o}hmov\'{a}, Dalf\'{o}, and Huemer~\cite{BoDaHu14} proved by using a
combinatorial approach.
Moreover, $n_k$ is also the number of ternary length-2 squarefree words of length
$k+4$ (that is, words on a three-letter alphabet that do not contain an adjacent
repetition of any subword of length $\le 2$); see the sequence A022089 in the On-Line Encyclopedia of Integer Sequences~\cite{Sl}.

In fact our method allows us to generalize this result and, for instance, derive a
formula for the order of $L^k(CK(d,4))$ for any value of the degree $d\ge 2$. To
this end, it is easy to see that a quotient digraph of $CK(d,4)$ for $d>2$ is as
shown in Figure~\ref{fig:quocient-donut-color}$(c)$, where now we have to
distinguish four classes of vertices.
%\begin{figure}[t]
%    \vskip-.5cm
%    \begin{center}
%        \includegraphics[width=8cm]{quocientCK(d,4).pdf}
%    \end{center}
%    \vskip-6.75cm
%	\caption{The quotient digraph of the cyclic Kautz digraph $CK(d,4)$.}
%	\label{fig:quotientCK(d,4)}
%\end{figure}
Then, the corresponding quotient matrix is
$$
\B=\left(
  \begin{array}{cccc}
    1 & d-1 & 0 & 0 \\
    0 & 0   & 1 & d-2\\
    1 & d-1 & 0 & 0 \\
    0 & 0   & 1 & d-2
  \end{array}
\right),
$$
and it has minimal polynomial is $m(x)=x^3-(d-1)x^2-x$.
In turn, this leads to the recurrence formula $n_k=(d-1)n_{k-1}+n_{k-2}$, with
initial values $n_0=d^4+d$ and $n_1=d^5-d^4+d^3+2d^2-d$, which are
computed by using \eqref{n_l} with the vector
\begin{align*}
  \s & =(|V_1|,|V_2|,|V_3|,|V_4|) \\
  & =((d+1)d, (d+1)d(d-1), (d+1)d(d-1), (d+1)d(d-1)(d-2)).
\end{align*}
Solving the recurrence, we get the closed formula
$$
n_k =\frac{2^k d}{\sqrt{\Delta}}
\left(\frac{
(d^2+d)\sqrt{\Delta}-d^3-d-2
%\left(\frac{2}{1-d-\sqrt{\Delta}}\right)^k
}
{(1-d-\sqrt{\Delta})^{k+1}}
+\frac{(d^2+d)\sqrt{\Delta}+d^3+d+2
%\left(\frac{2}{1-d+\sqrt{\Delta}}\right)^k
}
{(1-d+\sqrt{\Delta})^{k+1}}\right),
$$
where $\Delta=d^2-2d+5$ and, hence, $n_k$ is an increasing sequence.

\subsection{Unicyclic digraphs}
\label{sec:ex-unicyclic}

\begin{figure}[t]
    %\vskip-.5cm
    \begin{center}
        \includegraphics[width=10cm]{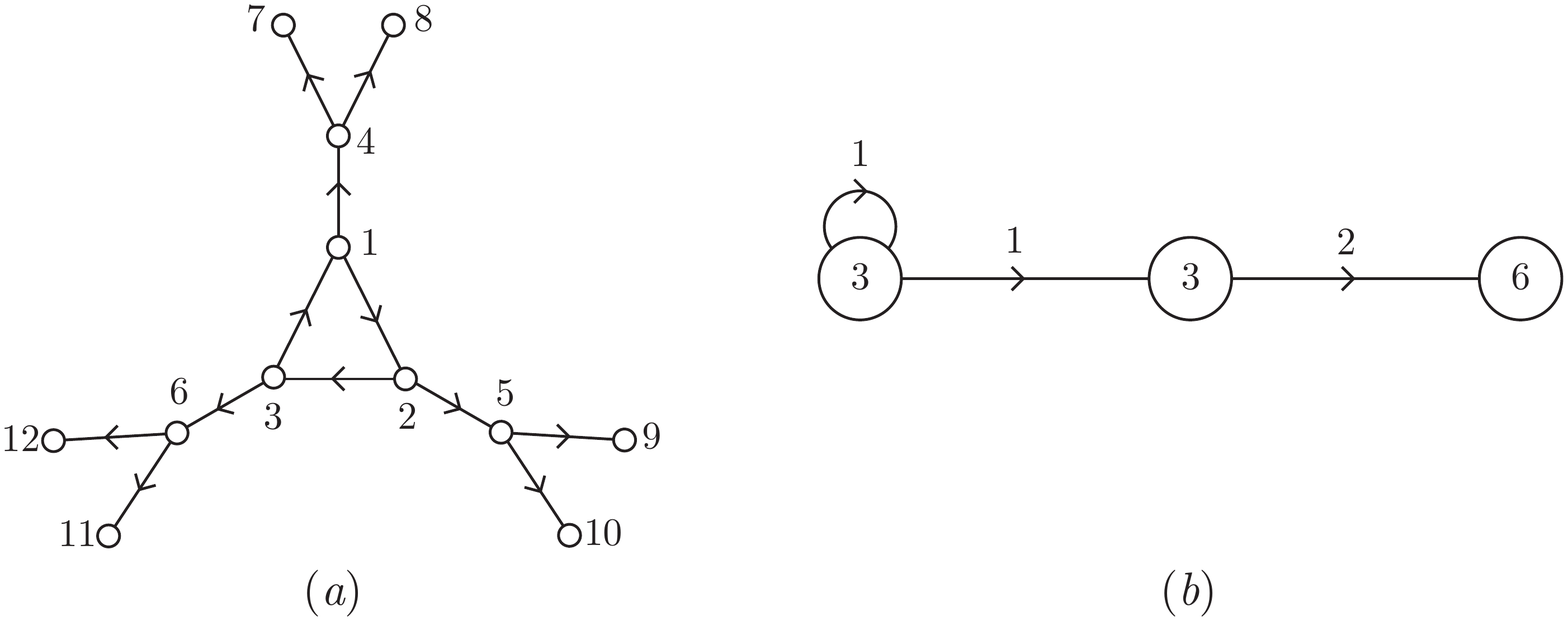}
    \end{center}
    \vskip-3.5cm
	\caption{The unicyclic digraph $G_{3,2}$ and its quotient digraph.}
	\label{fig:unicyclic}
\end{figure}

A unicyclic digraph is a digraph with exactly one (directed) cycle. As usual, we
denote a cycle on $n$ vertices by $C_n$. For example, consider the digraph
$G_{n,d}$, obtained by joining to every vertex of $C_n$ one `out-tree' with $d$
leaves (or `sinks'), as shown in Figure~\ref{fig:unicyclic}$(a)$ for the case
$G_{3,2}$. This digraph has the regular partition $\pi=(V_1,V_2,V_3)$, where $V_1$ is the set of vertices of the cycle, $V_2$ the central vertices of the trees,  and $V_3$ the set of leaves. (In the figure $V_1=\{1,2,3\}$, $V_2=\{4,5,6\}$, and $V_3=\{7,8,9,10,11,12\}$).
This partition gives the quotient digraph $\pi(G)$ of
Figure~\ref{fig:unicyclic}$(b)$, and the
quotient matrix
$$
\B=\left(
  \begin{array}{ccc}
    1 & 1 & 0 \\
    0 & 0 & d \\
    0 & 0 & 0 \\
  \end{array}
\right),
$$
with minimal polynomial $m(x)=x^3-x^2$. Then, by Theorem \ref{maintheo}, the order
of $L^k(G)$ satisfies the recurrence $n_k=n_{k-1}$ for $k\ge 0$, since
$\s(\B^{k}-\B^{k-1})\j^{\top}$ $=0$ for $k=1,2$, where $\s=(n,n,nd)$. Thus, we conclude that all the iterated line digraphs $L^k(G)$ have constant order
$n_k=n_0=n(d+2)$, that is, $n_k$ tends to a positive constant.
(In fact, this is because in this case $L(G)$---and, hence, $L^k(G)$---is
isomorphic to $G$.)

\subsection{Acyclic digraphs}
\label{sec:ex-acyclic}

\begin{figure}[t]
    \begin{center}
        \includegraphics[width=9cm]{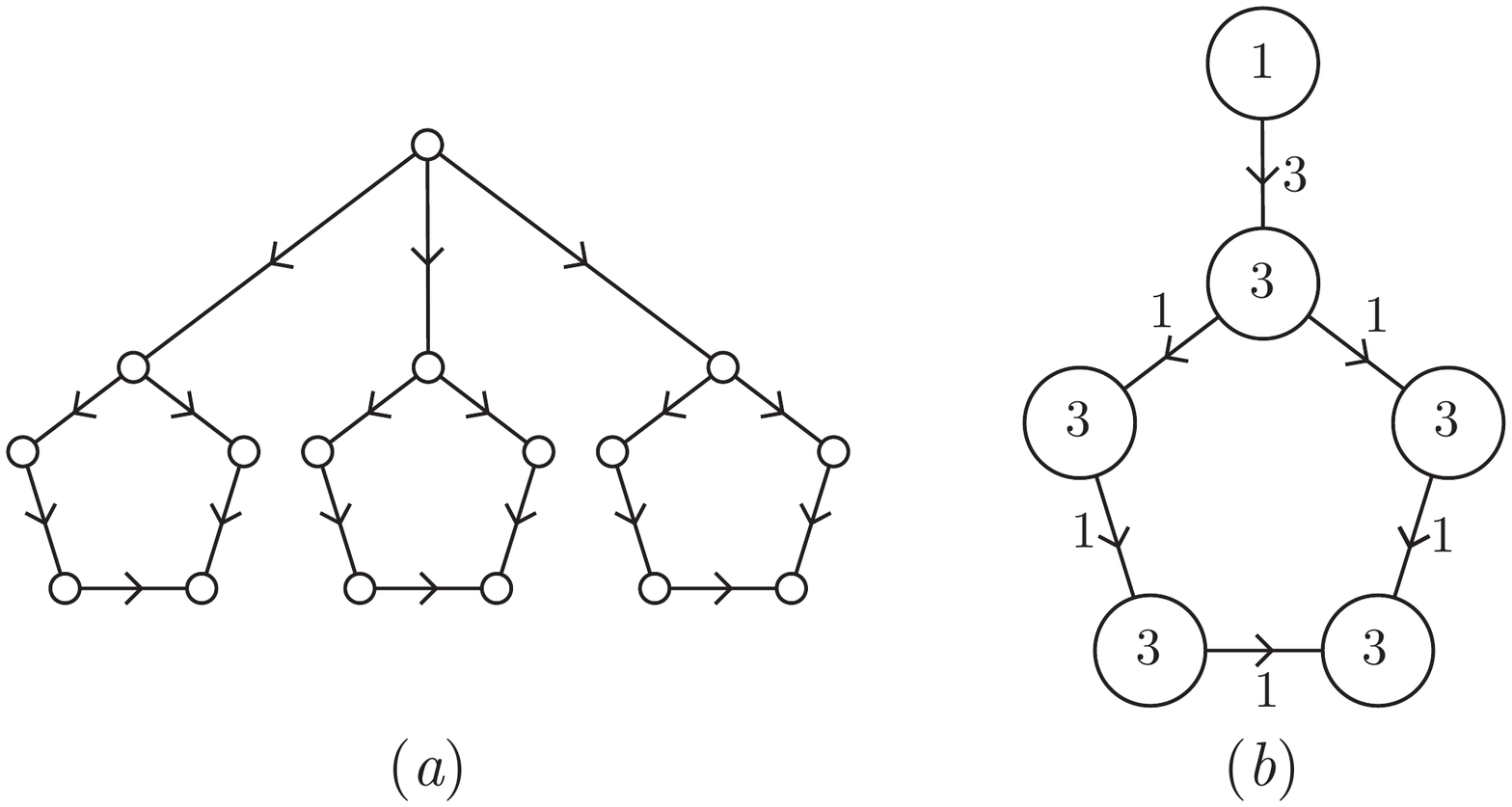}
    \end{center}
      \vskip-8.5cm
  \caption{An acyclic digraph and its quotient digraph.}
	\label{fig:acyclic}
\end{figure}

Finally, let us consider an example of an acyclic digraph, that is, a digraph
without directed cycles, such as the digraph $G$ of Figure~\ref{fig:acyclic}$(a)$.
Its quotient digraph is depicted in Figure~\ref{fig:acyclic}$(b)$, with quotient
matrix
$$
\B=\left(
  \begin{array}{cccccc}
    0 & 3 & 0 & 0 & 0 & 0 \\
    0 & 0 & 1 & 1 & 0 & 0 \\
    0 & 0 & 0 & 0 & 1 & 0 \\
    0 & 0 & 0 & 0 & 0 & 1 \\
    0 & 0 & 0 & 0 & 0 & 1 \\
    0 & 0 & 0 & 0 & 0 & 0 \\
  \end{array}
\right),
$$
and minimal polynomial $m(x)=x^5$. This indicates that $n_k=0$ for every $k\ge 5$
(as expected, because $G$ has not walks of length larger than or equal to $5$). Moreover, from \eqref{n_l}, the first values are
$n_0=16$, $n_1=18$, $n_2=15$, $n_3=9$, and $n_4=3$.

\vskip 1cm
\noindent{\large \bf Acknowledgments.}  This research is supported by the
{\em Ministerio de Econom\'{\i}a y Competitividad} and the {\em European Regional
Development Fund} under project MTM2011-28800-C02-01, and by the Catalan Government
under project 2014SGR1147.

%%%%%%%%%%%%%%%%%%%%%%%%%%%%%%%%%%%%%%%%%%%%%%%%
%Bibliografia
%%%%%%%%%%%%%%%%%%%%%%%%%%%%%%%%%%%%%%%%%%%%%%%%

\end{document}